\documentclass[12pt]{article}
\RequirePackage[colorlinks,citecolor=blue,urlcolor=blue,linkcolor=blue]{hyperref}
\hypersetup{
colorlinks = true,
citecolor=blue,
urlcolor=blue,
linkcolor=blue,
pdfauthor = {Alexey Kuznetsov},
pdfkeywords = {Cauchy Residue Theorem,
Jacobi elliptic functions, theta functions,  indeterminate moment problem, Nevanlinna parametrization},
pdftitle = {A direct evaluation of an integral of Ismail and Valent},
pdfpagemode = UseNone
}
\usepackage{graphicx,xspace,colortbl}
\usepackage{amsmath,amsthm,amsfonts,amssymb}
\usepackage{color}
\usepackage{enumerate}
\usepackage{fancybox}
\usepackage{epsfig}
\usepackage{subfig}
\usepackage{pdfsync}
    \def\qed{\hfill$\sqcap\kern-8.0pt\hbox{$\sqcup$}$\\}
    \def\beq{\begin{eqnarray}}
    \def\eeq{\end{eqnarray}}
    \def\beqq{\begin{eqnarray*}}
    \def\eeqq{\end{eqnarray*}}

\DeclareMathOperator{\re}{Re}
\DeclareMathOperator{\im}{Im}

    \def\r{{\mathbb R}}
    \def\c{{\mathbb C}} 
    	
    \def\d{{\textnormal d}}
    \def\i{{\textnormal i}}

    \newmuskip\pFqskip
\mathchardef\pFcomma=\mathcode`,

\newtheorem{theorem}{Theorem}

\newtheorem{corollary}{Corollary}
\theoremstyle{definition}

\title{A direct evaluation of an integral of Ismail and Valent}
\author{
{Alexey Kuznetsov
\footnote{Dept. of Mathematics and Statistics,  York University,
4700 Keele Street, Toronto, ON, M3J 1P3, Canada.  \newline
E-mail:  kuznetsov@mathstat.yorku.ca  } 
 }}
 
 \date{\today}

\begin{document}
\maketitle

\begin{abstract} 
We give a direct evaluation of a curious integral identity, which follows from the work of 
Ismail and Valent on the Nevanlinna parametrization of solutions to a certain indeterminate moment problem. 
\end{abstract}

{\vskip 0.15cm}

\section{Introduction and the main results}

At the International Conference on Orthogonal Polynomials and q-Series, which was held in Orlando in May 2015 in celebration
of the 70$^{\textnormal{th}}$ birthday of Mourad Ismail, Dennis Stanton gave a plenary talk titled ``A small slice of Mourad's work". One of the topics in that talk was about ``the mystery integral of Mourad Ismail": a curious integral that has first appeared in the paper 
\cite{IV1998} by Ismail and Valent (see also \cite{Chen1998} for a special case). To present this integral, we fix $k\in (0,1)$ and denote by 
$$
K(k)=\frac{\pi}{2} \times  {}_2F_1(\tfrac{1}{2},\tfrac{1}{2};1;k^2)
$$   
the complete elliptic integral of the first kind. We also denote $K=K(k)$, $k'=\sqrt{1-k^2}$ and $K'=K(k')$. The ``mystery integral", which Dennis Stanton referred to in his talk, is the following one:
\begin{equation}\label{mystery_integral}
\frac{1}{2}\int_{\r} \frac{\d x}{\cos(\sqrt{x}K)+\cosh(\sqrt{x}K')}=1. 
\end{equation}
This is essentially formula (1.16) in \cite{IV1998}, after correcting the typo -- an extra factor of $1/2$ multiplying $\sqrt{x}$. 

The identity \eqref{mystery_integral} is indeed rather unusual and mysterious. 
First of all, there is a free parameter $k$ that affects in a non-trivial way the integrand in the left-hand side, but 
the right-hand side stays constant. The appearance of the complete elliptic integral in combination with trigonometric and hyperbolic functions is also uncommon. The integrand looks very simple (after all, it only has two elementary
trigonometric functions and well-known complete elliptic integrals), but this simplicity is deceptive. In fact, this is the most striking feature of this integral: it is not at all clear how to prove such a result directly. The identity 
\eqref{mystery_integral} follows   
 as a by-product of explicit computations related to Nevanlinna parametrization in a certain indeterminate moment problem,
 see \cite{IV1998}.
 It is the goal of this paper is to evaluate the integral in \eqref{mystery_integral} directly, without using the theory of the indeterminate moment problem.

Our main result is the following theorem, which gives a more general statement than \eqref{mystery_integral}. As we will see later, this theorem allows to compute explicitly  all moments of the probability measure appearing in \eqref{mystery_integral}. 
In what follows we will be working with Jacobi elliptic functions; we refer the reader to \cite{Olver}[Chapter 22] for their definition and various properties. 

\begin{theorem}\label{theorem3}
Assume that $k \in (0,1)$ and denote $k'=\sqrt{1-k^2}$, $K=K(k)$ and $K'=K(k')$. Then for $u \in \c$ satisfying $|\re(u)|<K$ and $|\im(u)|<K'$ we have
\begin{equation}\label{theorem3_eqn1}
\frac{1}{2}\int_{\r} \frac{\sin(\sqrt{x}u)}{{\sqrt{x}}} \times \frac{\d x}
{\cos(\sqrt{x}K)+\cosh(\sqrt{x}K')}=
 \frac{{\textnormal{sn}}(u,k)}{{\textnormal{cd}}(u,k)}.
\end{equation}
\end{theorem} 
\begin{proof}
Our plan is to establish \eqref{theorem3_eqn1} for $u=v(K+\i K')/2$ with $v\in (-1,1)$,
and then apply an analytic continuation argument to extend this result to other values of $u$. Thus, we fix $v \in (-1,1)$ and we denote  
\begin{equation}\label{def_I}
I:=\frac{1}{2}\int_{\r} \frac{\sin(\sqrt{x}v(K+\i K')/2)}{{\sqrt{x}}} \times \frac{\d x}
{\cos(\sqrt{x}K)+\cosh(\sqrt{x}K')}.
\end{equation}

Our first step is to change the variable of integration $x=(2z/K)^2$ in \eqref{def_I}. This implies $z=K\sqrt{x}/2$, and the original contour of integration $\r \ni x$ is mapped into the contour $ L \ni z$, where $L$ consists of two half-lines 
$(+\i \infty, 0] \cup [0,+\infty)$, see Figure \ref{fig1}. This contour is traversed in the direction $+\i \infty \to 0 \to +\infty$. After this change of variables we obtain
\begin{equation}\label{eqn_I2}
I=\frac{2}{K} \int_{L} \frac{\sin(zv(1+\tau)) \d z}{\cos(2z)+\cos(2z\tau)},
\end{equation}
where we have denoted 
\begin{equation*}
\tau:=\i  \frac{K'}{K}. 
\end{equation*}
Using trigonometric sum-to-product identity we rewrite \eqref{eqn_I2} in the form
\begin{equation}\label{eqn_I3}
I=\frac{1}{K}\int_{L} \frac{\sin(zv(1+\tau)) \d z}{\cos(z(1+\tau)) \cos(z(1-\tau))}.
\end{equation}

Our second step is to compute the integral in \eqref{eqn_I3} via Cauchy's Residue Theorem. 
Note that the integrand in \eqref{eqn_I3} is a meromorphic function that has only simple poles. Only the poles lying in the 
first quadrant are of importance to us, and these are given by  
$$
z_n:=\pi (n-1/2)\frac{1}{1-\tau},  \;\;\; n \in {\mathbb N}.
$$ 
We also introduce the following notation: 
$$
w_n:=\pi n \frac{1}{1-\tau}, \; {\textnormal{ and }} \;
t:=\frac{1+\tau}{1-\tau}.
$$
It is clear that the points $z_n$ all lie on a ray in the first quadrant, and $w_n$ are the midpoints between $
z_n$ and $z_{n+1}$, see Figure \ref{fig1}. 
\begin{figure}[t]
\centering
\includegraphics[height =9cm]{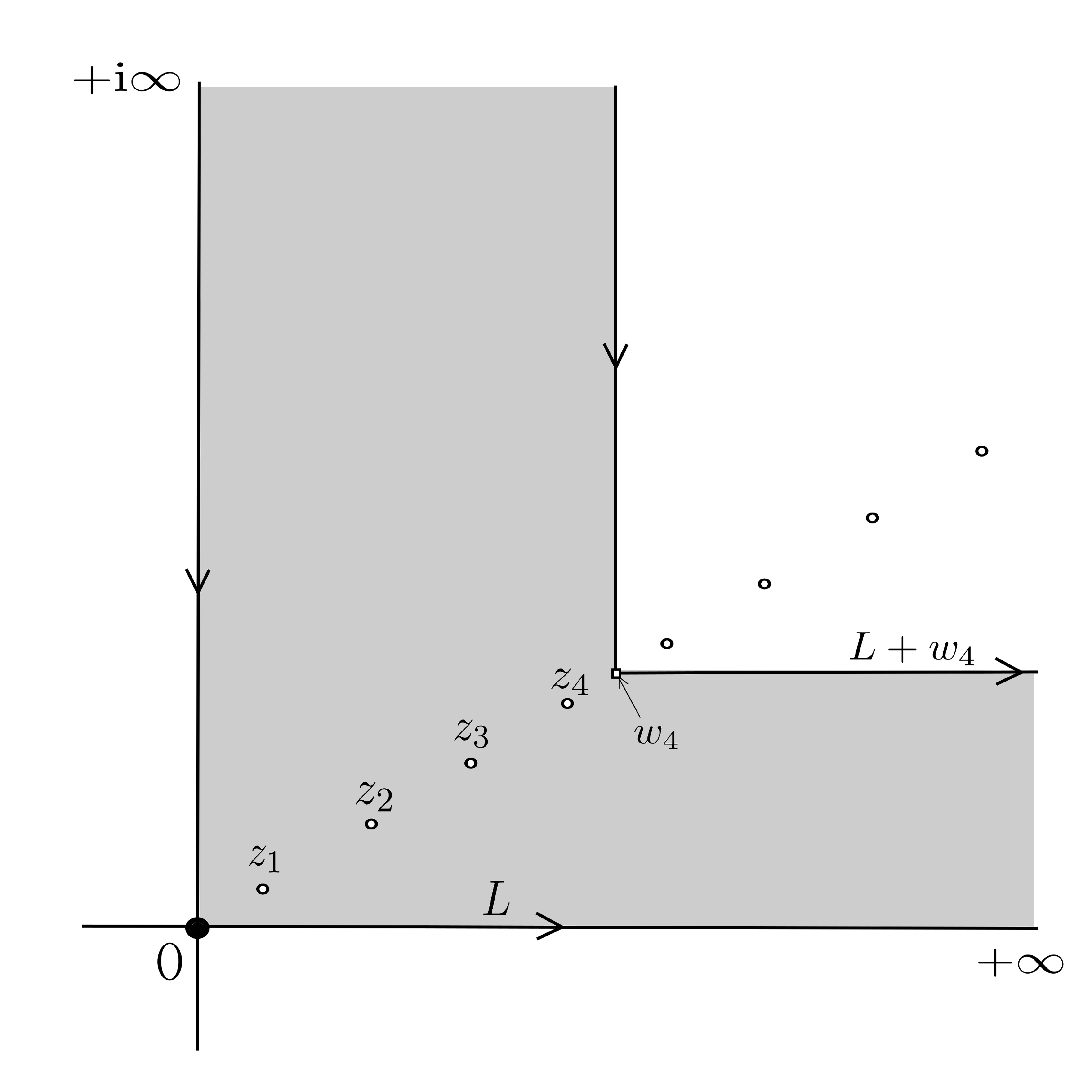} 
\caption{Here we see the contour of integration $L$ that is used in formula \eqref{eqn_I3}, the poles at points $\{z_n\}_{n\ge 1}$ and the shifted contour of integration
$L+w_N$ when $N=4$.}
\label{fig1}
\end{figure}

We choose $N \in {\mathbb N}$ and we shift the contour of integration $L \mapsto L+w_N$, apply the Cauchy Residue Theorem 
and take into account the residues at $z=z_n$, $1\le n \le N$:
\begin{equation}\label{I_to_I_n}
I= 2 \pi \i \times \frac{1}{K} \sum\limits_{n=1}^{N} r_n(v)
+I_N,
\end{equation}
where we have denoted 
\begin{equation}\label{def_I_m}
I_N:= \frac{1}{K} \int_{L+w_N} \frac{\sin(vz (1+\tau)) \d z}{\cos(z(1+\tau)) \cos(z(1-\tau))},
\end{equation}
and
\begin{align}\label{formula_for_rn}
r_n(v):=  {\textnormal{Res}} \Big( \frac{\sin(vz (1+\tau)) }{\cos(z(1+\tau)) \cos(z(1-\tau))}\; \Big \vert \; z=z_n \Big )=\frac{(-1)^{n}}{1-\tau} 
\frac{\sin(\pi (n-1/2) t v)}
{\cos(\pi (n-1/2) t)}.
\end{align}
Note that shifting the contour of integration is justified, since the integrand decays exponentially fast as $z\to \infty$ in the area between the two contours $L$ and $L+w_N$ (this is the gray area in Figure \ref{fig1}). 

Next we plan to show that $I_N \to 0$ as $N\to +\infty$. By changing the variable of integration 
$z=w_N+w$ we rewrite \eqref{def_I_m} in the form
\begin{equation}\label{I_m2}
I_N=\frac{(-1)^N}{K} \int_{L} 
\frac{\sin(v w (1+\tau)+\pi N t v) \d w}{\cos(w(1+\tau)+\pi N t) \cos(w(1-\tau))}.
\end{equation}
Note that $\arg(1+\tau) \in (0,\pi/2)$ and $\arg(1-\tau)\in (-\pi/2,0)$, therefore 
\begin{equation*}
\im(t)>0 \; {\textnormal{ and }} \; \im(w(1+\tau))>0, \;\;\; {\textnormal{ for all }} \; w \in L.
\end{equation*}
The above conditions imply $\im(w(1+\tau)+\pi N t)>10$ for all $w \in L$ and all $N$ large enough. 
Using the following trivial estimates 
\begin{align*}
&|\sin(a)|<\exp(\im(a)), \;\qquad  {\textnormal{ for }} \; a\in \c \; {\textnormal{ such that }} \; \im(a)>0,\\
&|\cos(a)|>\exp(\im(a))/10, \;\;\;\; {\textnormal{ for }} \; a\in \c \; {\textnormal{ such that }} \; \im(a)>10,
\end{align*}
we conclude that for all $z\in L$ and all $N$ large enough we have
\begin{align}\label{estimate}
\Big| \frac{\sin(vz (1+\tau)+\pi N t v)}{\cos(z(1+\tau)+\pi N t)} \Big|
&< 10 \exp((v-1) \im(z(1+\tau))+\pi N \im(t)(v-1))\\ \nonumber
&<10 \exp( \pi N  \im(t)(v-1)).
\end{align}

Combining \eqref{I_m2} and \eqref{estimate} we obtain
$$
|I_N|\le  \frac{10}{K} \exp(\pi N  \im(t)(v-1)) \times  \int_{L} 
\frac{|\d z|}{|\cos(z(1-\tau))|},
$$
and the right-hand side converges to zero as $N \to +\infty$ (recall that $v-1<0$ and $\im(t)>0$). 
The above result and formula \eqref{I_to_I_n} imply the following identity
\begin{equation}\label{I_inf_sum}
I=\frac{2\pi \i}{K(1-\tau)} \sum\limits_{n=1}^{\infty}
(-1)^{n}
\frac{\sin(\pi (n-1/2) t v)}
{\cos(\pi (n-1/2) t)}.
\end{equation}

Our third step is to express the infinite sum in \eqref{I_inf_sum} in terms of Jacobi elliptic functions. 
Formula 22.11.5 in \cite{NIST} tells us that  
$$
\pi \sum\limits_{n=1}^{\infty} (-1)^n  \frac{\sin(\pi (n-1/2) z)}{\cosh(\pi (n-1/2) K'/K)}
=- K k k'
\frac{{\textnormal{sn}}(K z,  k)}
{{\textnormal{dn}}(K z,  k)}.
$$
Using this result and formulas 22.2.4 and 22.2.6 in \cite{NIST}, which express Jacobi elliptic functions in terms of theta functions, and 
formulas 22.2.2 in \cite{NIST}, which express the constants $k$, $k'$ and $K$ in terms of theta functions, we arrive
at the following expression
\begin{equation}\label{I_theta_1}
I=-\frac{\pi \i}{K(1-\tau)} 
\theta_2(0,t) \theta_4(0,t)
\frac{\theta_1( \pi tv/2,t)}
{\theta_3( \pi  \pi tv/2,t)}.
\end{equation}
Here $\theta_i(z,t)$ (with $z,t\in \c$ and $\im(t)>0$) are the four theta functions, as defined in formulas 20.2.1-20.2.4 in \cite{NIST}.

Our plan is to apply transformations of theta functions $\theta_i(\cdot,t)$ with respect to the parameter $t$ so that 
we obtain an expression involving $\theta_i(\cdot, \tau)$. This will be done in four steps, and the sequence of transformations is summarized here
\begin{align*}
&t=\frac{1+\tau}{1-\tau} \longmapsto t_1:=t+1=\frac{2}{1-\tau} \longmapsto
t_2:=t_1/2=\frac{1}{1-\tau}  \longmapsto \\ 
 & \qquad \qquad \qquad \; \, t_3:=-1/t_2=\tau-1 \longmapsto t_4:=t_3+1=\tau. 
\end{align*}

\vspace{0.25cm}
\noindent
{\bf Transformation 1, $t\longmapsto t_1$:} 
We apply formulas 20.7.26-20.7.29 in \cite{NIST}
to the expression in \eqref{I_theta_1} and obtain
\begin{equation}\label{I_theta_2}
I=-\frac{\pi}{K(1-\tau)} \theta_2(0,t_1) \theta_3(0,t_1)
\frac{\theta_1(\pi tv/2,t_1)}
{\theta_4(\pi tv/2,t_1)}.
\end{equation}

\vspace{0.25cm}
\noindent
{\bf Transformation 2, $t_1\longmapsto t_2$:} 
We apply formulas  20.7.11-20.7.12 in \cite{NIST} to the expression in \eqref{I_theta_2} and obtain
\begin{equation}\label{I_theta_3}
I=-\frac{\pi}{2K(1-\tau)} \theta_2(0,t_2)^2 
\frac{\theta_1( \pi tv/4,t_2)\theta_2( \pi tv/4,t_2)}
{\theta_3( \pi tv/4,t_2)\theta_4(\pi tv/4,t_2)}.
\end{equation}

\vspace{0.25cm}
\noindent
{\bf Transformation 3, $t_2\longmapsto t_3$:} 
We apply formulas  20.7.30-20.7.33 in \cite{NIST} to the expression in \eqref{I_theta_3} and obtain
\begin{equation}\label{I_theta_4}
I=-\frac{\pi}{2K} \theta_4(0,t_3)^2 
\frac{\theta_1( \pi tv t_3/4,t_3)\theta_4( \pi tv t_3/4,t_3)}
{\theta_3( \pi tv t_3/4,t_3)\theta_2( \pi tv t_3/4,t_3)}.
\end{equation}

\vspace{0.25cm}
\noindent
{\bf Transformation 4, $t_3\longmapsto \tau$:} 
Finally, we apply formulas 20.7.26-20.7.29 in \cite{NIST} to the expression in \eqref{I_theta_4} and obtain
\begin{equation}\label{I_theta_5}
I=-\frac{\pi}{2K} \theta_3(0,\tau)^2 
\frac{\theta_1(  \pi tv t_3/4,\tau)\theta_3( \pi tv t_3/4,\tau)}
{\theta_4( \pi tv t_3/4,\tau)\theta_2( \pi tv t_3/4,\tau)}.
\end{equation}

Now we check that $\pi tv t_3/4=-\pi v (1+\tau)/4$, we apply formulas 22.2.4 and 22.2.8 in \cite{NIST} 
and rewrite the expression in \eqref{I_theta_5} in terms of Jacobi elliptic functions, which gives us
\begin{equation*}
I=\frac{{\textnormal{sn}}(vK(1+\tau)/2,k)}{{\textnormal{cd}}(vK(1+\tau)/2,k)}.
\end{equation*}
Recalling our definition of $I$ in \eqref{def_I}, we see that we have established formula \eqref{theorem3_eqn1}
for $u=v(K+\i K')/2$ with $v\in (-1,1)$.

As the final step, we need to show that \eqref{theorem3_eqn1} holds true in the bigger region
$D:=\{ u \in \c \; : \; |\re(u)|<K, \; |\im(u)|<K'\}$. 
This is easy to achieve by analytic continuation. Indeed, the integral in the left hand side of \eqref{theorem3_eqn1} converges 
 absolutely and uniformly for all $u$ on compact subsets of $D$, thus this integral defines an analytic function on $D$. The right-hand side of  
 \eqref{theorem3_eqn1} is also analytic in $D$ (one can check this by locating the poles of 
 ${\textnormal{sn}}(u,k)$ and the roots of ${\textnormal{cd}}(u,k)$, see Tables 22.4.1 and 22.4.2 in \cite{NIST}). Thus, by analytic continuation, the identity \eqref{theorem3_eqn1} is valid not only for $u=v(K+\i K')/2$ with $v\in (-1,1)$, but for all $u\in D$. 
\end{proof}

By expanding both sides in \eqref{theorem3_eqn1} in Taylor series in $u$ we obtain the following result.
The mystery integral identity \eqref{mystery_integral} follows  by setting $n=0$ in the formula \eqref{moments} below. 
\begin{corollary}\label{corollary1}
With the notation of Theorem \ref{theorem3} we have
\begin{equation}\label{moments}
\frac{1}{2}\int_{\r} \frac{x^n \d x}
{\cos(\sqrt{x}K)+\cosh(\sqrt{x}K')}=
(-1)^n \times \frac{\d^{2n+1}}{\d u^{2n+1}}
 \frac{{\textnormal{sn}}(u,k)}{{\textnormal{cd}}(u,k)}
 \Big \vert_{u=0},
\end{equation}
for $n\ge 0$. 
\end{corollary}

\section{A more general version of the mystery integral}

In the paper \cite{IV1998} Ismail and Valent compute explicitly the functions $D(x)$ and $B(x)$ appearing in the Nevanlinna parametrization of the indeterminate moment problem, which has the same moments as in \eqref{moments}. These functions are  
\begin{align*}
D(x)&=-\frac{4}{\pi} \sin(\sqrt{x}K/2) \sinh(\sqrt{x} K'/2), \\
B(x)&=\frac{2}{\pi} \ln(k/k') \sin(\sqrt{x}K/2) \sinh(\sqrt{x} K'/2)
+\cos(\sqrt{x}K/2) \cosh(\sqrt{x} K'/2), 
\end{align*}
see formulas (4.16) and (4.17) in \cite{IV1998}. For $t \in \r$ and $\gamma>0$ we define 
$$
w(x;t,\gamma):=\frac{\gamma/\pi}{(D(x)-tB(x))^2+\gamma^2 B(x)^2}, \;\;\; x\in \r.  
$$
As was proved in \cite{IV1998} using the Nevanlinna parametrization and the theory of indeterminate moment problems, the measures $w(x;t,\gamma)\d x$ have the same moments for all $t\in \r$ and $\gamma>0$. One can check (after some tedious algebraic computations) that 
for all $x\in \r$
$$
w(x;t^*, \gamma^*)=
\frac{1/2}
{\cos(\sqrt{x}K)+\cosh(\sqrt{x}K')}, 
$$
provided that
$$
\gamma^*=\frac{4}{\pi(1+C^2)}, \;\;\; t^*=-C\gamma^*, \;
{\textnormal{ and }} \; C=\frac{2}{\pi} \ln(k/k'). 
$$

 Combining this fact with Corollary \ref{corollary1} we obtain the following result, which generalizes Theorem \ref{theorem3}. 

\begin{theorem}\label{theorem2}
Assume that $t\in \r$, $\gamma>0$, $k \in (0,1)$ and denote $k'=\sqrt{1-k^2}$, $K=K(k)$ and $K'=K(k')$. Then for $u \in \c$ satisfying $|\re(u)|<K$ and $|\im(u)|<K'$ we have
\begin{equation}\label{theorem2_eqn1}
\int_{\r} \frac{\sin(\sqrt{x}u)}{{\sqrt{x}}} \times 
w(x;t,\gamma) \d x=
 \frac{{\textnormal{sn}}(u,k)}{{\textnormal{cd}}(u,k)}.
\end{equation}
\end{theorem} 

It turns out that the integral in \eqref{theorem2_eqn1} also can be evaluated directly, without using the Nevanlinna parametrization or the theory of the indeterminate moment problem. To do this, one only needs to show directly that the measures
$ w(x;t,\gamma)\d x$ have the same moments. 
We plan to present this approach in a more general setting in 
 the forthcoming paper \cite{Kuznetsov}. 

\section*{Acknowledgements}

 The research was supported by the Natural Sciences and Engineering Research Council of Canada. 
  We would like to thank Mourad Ismail for helpful discussions.


\end{document}